\documentclass[11pt]{article}

\usepackage[latin1]{inputenc}
\usepackage{enumerate}
\usepackage[hang,flushmargin]{footmisc}
\usepackage{amsmath}
\usepackage{amsfonts}
\usepackage{amssymb}
\usepackage{amsthm}
\usepackage{subcaption}
\usepackage{tikz,tkz-graph}
\tikzstyle{vertex}=[circle,fill=black,inner sep=2pt]

\usepackage{fancyhdr}
\usepackage{lastpage}
\usepackage[pass,letterpaper]{geometry}

\newtheorem{theorem}{Theorem}[section]
\newtheorem*{theorem*}{Theorem}
\newtheorem{lemma}[theorem]{Lemma}
\newtheorem*{lemma*}{Lemma}
\newtheorem*{question*}{Question}
\newtheorem{corollary}[theorem]{Corollary}

\newtheorem*{proposition*}{Proposition}
\newtheorem{definition}[theorem]{Definition}

\newcommand{\abs}[1]{\left\vert#1\right\vert}
\newcommand{\fl}[1]{\left\lfloor#1\right\rfloor}
\newcommand{\de}[1]{\delta_{#1}}
\newcommand{\dei}[1]{\delta_{i_{#1}}}
\newcommand{\dej}[1]{\delta_{j_{#1}}}
\newcommand{\dd}{\delta}
\newcommand{\com}{\mathrel{\ooalign{$\hidewidth,\hidewidth$\cr$\phantom{>}$}}}
\newcommand{\twr}{{{\rm twr}}}

\begin{document}
\title{A note on the Erd\H os-Hajnal hypergraph Ramsey problem}

\author{Dhruv Mubayi\thanks{Department of Mathematics, Statistics, and Computer Science, University of Illinois, Chicago, IL, 60607 USA.  Research partially supported by NSF grants DMS-1300138 and DMS-1763317. Email: {\tt mubayi@uic.edu}} \and Andrew Suk\thanks{Department of Mathematics, University of California at San Diego, La Jolla, CA, 92093 USA. Research supported by NSF CAREER award DMS-1800746 and by an Alfred Sloan Fellowship. Email: {\tt asuk@ucsd.edu}.} \and Emily Zhu\thanks{Department of Mathematics, University of California at San Diego, La Jolla, CA, 92093 USA.  Email: {\tt e9zhu@ucsd.edu}. }}
\date{}

\maketitle

\begin{abstract}

We show that there is an absolute constant $c>0$ such that the following holds.  For every $n > 1$, there is a 5-uniform hypergraph on at least $2^{2^{cn^{1/4}}}$ vertices with independence number  at most $n$, where every set of 6 vertices induces at most 3 edges.  The double exponential growth rate for the number of vertices is sharp. By applying a stepping-up lemma established by the first two authors, analogous sharp results  are proved for $k$-uniform hypergraphs.  This answers the penultimate open case of a conjecture in Ramsey theory posed by Erd\H os and Hajnal in 1972.
\end{abstract}

\section{Introduction}

The Ramsey number $r_k(s,n)$ is the minimum integer $N$ such that for any red/blue coloring of the $k$-tuples of $[N] = \{1,2,\ldots, N\}$, there is either a set of $s$ integers with all of its $k$-tuples colored red, or a set of $n$ integers with all of its $k$-tuples colored blue.   Estimating $r_k(s,n)$ is a fundamental problem in combinatorics and has been extensively studied since 1935.  For graphs, classical results of Erd\H os \cite{E47} and Erd\H os and Szekeres \cite{ES35} imply that $2^{n/2} < r_2(n,n) < 2^{2n}$. While small improvements have been made in
both the upper and lower bounds for $r_2(n, n)$ (see \cite{conlon,spencer}), the constant factors in the exponents have not changed over the last 75 years.

Unfortunately for 3-uniform hypergraphs, there is an exponential gap between the best known upper and lower bounds for $r_3(n,n)$.  Namely, Erd\H os, Hajnal, and Rado \cite{EHR,ER} showed that 

$$2^{cn^2} < r_3(n,n) < 2^{2^{c'n}},$$

where $c$ and $c'$ are absolute constants.  For $k \geq 4$, their results also imply an exponential gap between the lower and upper bounds for $r_k(n, n)$, 

$$\twr_{k-1}(cn^2) < r_k(n,n) < \twr_k(c'n),$$

\noindent  where the \emph{tower function} is defined recursively as $\twr_1(x) = x$ and $\twr_{i + 1} = 2^{\twr_i(x)}.$   Determining the tower growth rate of $r_k(n,n)$ is one of the most central problems in extremal combinatorics.  Erd\H os, Hajnal, and Rado conjectured that the upper bound is closer to the truth, namely $r_k(n,n) = \twr_k(\Theta(n))$, and Erd\H os offered a \$500 reward for a proof (see~\cite{Chung}).

\emph{Off-diagonal} Ramsey numbers $r_k(s,n)$ have also been extensively studied.  Here, $k$ and $s$ are fixed constants and $n$ tends to infinity.  It follows from well-known results that $r_2(s,n)  = n^{\Theta(1)}$  (see \cite{AKS,B,BK,ER} for the best known bounds), and for 3-uniform hypergraphs, $r_3(s,n) = 2^{n^{\Theta(1)}}$  (see~\cite{CFS} for the best known bounds). 

For $k> 3$, Erd\H os, Hajnal, and Rado showed that $r_k(s,n) \leq \twr_{k-1}(n^{c})$ where $c = c(k,s)$, and Erd\H os and Hajnal conjectured that this bound is the correct tower growth rate. In \cite{MS3}, the first two authors verified the conjecture for $s \geq k + 2$, and for the last case $s = k+1$, they showed that $r_k(k + 1,n) \geq \twr_{k-2}(n^{c\log n})$.  Hence, there remains an exponential gap between the best known lower and upper bounds for $r_k(k + 1,n)$ for $k \geq 4$.  

Due to our lack of understanding of $r_k(k+1,n)$,  Erd\H os and Hajnal in \cite{EH72} introduced the following more general function (their notation was different).

\begin{definition}
For integers $2\leq k  <n$ and $2 \leq t \leq k+1$, let $r_k(k+1,t;n)$ be the minimum $N$ such that for every red/blue coloring of the $k$-tuples of $[N]$, there is a set of $k+1$ integers with at least $t$ of its $k$-tuples colored red, or a set of $n$ integers with all of its $k$-tuples colored blue.  
\end{definition} 
 
Clearly $r_k(k + 1,1;n) =n$ and $r_k(k + 1,k + 1;n) = r_k(k + 1,n)$.  For each $t \in \{2,\ldots, k\}$, Erd\H os and Hajnal \cite{EH72} showed that $r_k(k + 1,t;n) < \twr_{t-1}(n^{\Theta(1)})$ and conjectured that

\begin{equation}\label{EHbound}
r_k(k + 1,t;n) = \twr_{t-1}(n^{\Theta(1)}).
\end{equation}

\noindent   This is known to be true for $k \leq 3$ and for $t \leq 3$ \cite{EH72}. When $k \geq 5$, the first two authors \cite{MS4} verified (\ref{EHbound}) for all $3 \leq t \leq k-2$.  Our main result verifies (\ref{EHbound}) for  $t = k-1$, which is one of the last two remaining cases.

\begin{theorem}\label{main1}
For $k\geq 4$, we have $r_k( k+1,k-1;n) = \twr_{k-2}(n^{\Theta(1)})$.
\end{theorem}

This significantly improves the previous best known lower bound for $r_k(k+1,k-1;n)$, which was one exponential less than above (see \cite{MS4}).  This also immediately implies the following new lower bound for $r_k(k + 1,k;n)$, which is now one exponential off from the upper bound obtained by Erd\H os and Hajnal.

\begin{corollary}
For $k\geq 4$, we have $r_k( k+1,k;n) > \twr_{k-2}(n^{\Theta(1)})$.
\end{corollary}

Finally, let us point out that Erd\H os and Hajnal conjectured that the tower growth rate for both $r_k(k + 1,k;n)$ and the classical Ramsey number $r_k(k+1,n)$ are the same.  Thus, verifying (\ref{EHbound}) for $r_k(k + 1,k;n)$ would determine the tower height   for $r_k(k + 1,n)$.

We develop several crucial new ingredients to the stepping up method in our construction, for example, part (1) of Lemma~\ref{lem:graphcolor}, and on page 8, analyzing sequences of local maxima. It is plausible that these new ideas can be further enhanced to determine the tower height of $r_k(k + 1,n)$.

\section{Proof of Theorem \ref{main1}}

In \cite{MS3}, the first two authors proved the following.

\begin{theorem}[Theorem 7 in \cite{MS3}]\label{stepk}

For $k\geq 6$ and $t\geq 5$, we have $r_k(k+1,t;2kn) > 2^{r_{k-1}(k,t-1;n)-1}.$

\end{theorem}

In what follows, we will prove the following theorem.  Together with Theorem \ref{stepk}, Theorem \ref{main1} quickly follows.

\begin{theorem}

There is an absolute constant $c> 0$ such that $r_5(6,4;n) > 2^{2^{cn^{1/4}}}$.

\end{theorem}

\subsection{A double exponential lower bound for $r_5(6,4;n)$}

In this section, we begin with a graph coloring with certain properties which we will later use to define a two-coloring of the edges of a 5-uniform hypergraph.

\begin{lemma}
\label{lem:graphcolor}
For $n \geq 6$, there is an absolute constant $c > 0$ such that the following holds. There exists a red/blue coloring $\phi$ of the pairs of $\{0,1,\dots,\fl{2^{cn}}\}$ such that:
\begin{enumerate}
\item There are no 3 disjoint $n$-sets $A, B, C \subset \{0,1,\dots,\fl{2^{cn}}\}$ with the property that there is a bijection $f: B \to C$ such that for any $a \in A, b \in B$, at least one of $\phi(a,b) = \text{red}$ or $\phi(a,f(b)) = \text{blue}$ occurs.
\item There is no $n$-set $A \subset \{0,1,\dots,\fl{2^{cn}}\}$ such that every 4-tuple $a_i,a_j,a_k,a_\ell \in A$ with $a_i < a_j < a_k < a_\ell$ avoids the pattern:
\[\phi(a_i,a_j) = \phi(a_j,a_k) = \phi(a_j,a_\ell) = \text{red}, \qquad \phi(a_i,a_k) = \phi(a_i,a_\ell) = \phi(a_k,a_\ell) = \text{blue}\]
\end{enumerate}
\end{lemma}
\begin{proof}
Set $N = \fl{2^{cn}}$, where $c$ is a sufficiently small constant that will be determined later. Consider a random 2-coloring of the unordered pairs of $\{0,1,\dots,N-1\}$ where each pair is assigned red or blue with equal probability independent of all other pairs. Then, the expected number of $A,B,C$ as in part 1 is at most 
\[\binom{N}{n}^3n!\left(\frac{3}{4}\right)^{n^2} < \frac{1}{3},\]
where the inequality holds by taking $c$ sufficiently small. This is since we pick each of the $n$-sets, one of $n!$ possible bijections from $B$ to $C$, and then there is a $\frac{3}{4}$ probability that we have the desired color pattern for each pair of $a\in A,b\in B$. 

We call a 4-tuple $a_i,a_j,a_k,a_\ell \in \{0,1,\dots,N-1\}$ with $a_i < a_j < a_k < a_\ell$ \emph{bad} if
\[\phi(a_i,a_j) = \phi(a_j,a_k) = \phi(a_j,a_\ell) = \text{red}, \qquad \phi(a_i,a_k) = \phi(a_i,a_\ell) = \phi(a_k,a_\ell) = \text{blue}\]
and \emph{good} otherwise. The probability that such a fixed 4-tuple is bad is $\frac{1}{2^6} = \frac{1}{64}$ and thus the probability that such a fixed 4-tuple is good is $\frac{63}{64}$. Now consider some fixed $n$-set $A \subset \{0,1,\dots,N-1\}$. We estimate the probability that $A$ contains no bad 4-tuple. Note that there exists a partial Steiner $(n,4,2)$-system $S$ on $A$, i.e.\ a 4-uniform hypergraph on the $n$-vertex set $A$ with the property that every pair of vertices is contained in at most one 4-tuple, with at least $c'n^2$ edges where $c'>0$ is some constant (e.g.\ see \cite{EH63}). Then, the probability that a 4-tuple in $A$ is good is at most the probability that every 4-tuple in $S$ is good. Since 4-tuples in $S$ are independent as no two 4-tuples have more than one vertex is common, the probability that every 4-tuple in $S$ is a good 4-tuple is at most $\left(\frac{63}{64}\right)^{c'n^2}$. Therefore, the expected number of $n$-sets $A$ with only good 4-tuples is at most
\[\binom{N}{n}\left(\frac{63}{64}\right)^{c'n^2} < \frac{1}{3},\]
again where we take $c$ sufficiently small. Thus, by Markov's inequality and the union bound, we  conclude that there is a 2-coloring $\phi$ with the desired properties.
\end{proof}

We will use this lemma to produce a coloring of a 5-uniform hypergraph. Given some natural number $N$, let $V = \{0,1,\dots,2^N-1\}$. Then for $v \in V$, we write $v = \sum_{i=0}^{N-1} v(i)2^i$ where $v(i) \in \{0,1\}$ for each $i$. For any $u \neq v$, we then let $\delta(u,v)$ denote the largest $i \in \{0,1,\dots,N-1\}$ such that $u(i) \neq v(i)$. We then have the following properties.

\textbf{Property I}: For every triple $u < v < w, \delta(u,v) \neq \delta(v,w)$.

\textbf{Property II}: For $v_1 < \cdots < v_r, \delta(v_1,v_r) = \max_{1\leq j \leq r-1} \delta(v_j,v_{j+1})$.

From Properties I and II, we also derive the following.

\textbf{Property III}: For every 4-tuple $v_1 < \cdots < v_4$, if $\delta(v_1,v_2) > \delta(v_2,v_3)$, then $\delta(v_1,v_2) \neq \delta(v_3,v_4)$. Note that if $\delta(v_1,v_2) < \delta(v_2,v_3)$, it is possible that $\delta(v_1,v_2) = \delta(v_3,v_4)$.

\textbf{Property IV}: For $v_1 < \cdots < v_r$, set $\delta_j = \delta(v_j,v_{j+1})$ for $j \in [r-1]$ and suppose that $\delta_1,\dots,\delta_{r-1}$ forms a monotone sequence. Then for every subset of $k$ vertices $v_{i_1},v_{i_2},\dots,v_{i_k}$ where $v_{i_1} < \cdots < v_{i_k}$, $\delta(v_{i_1},v_{i_2}), \delta(v_{i_2},v_{i_3}), \dots, \delta(v_{i_{k-1}},v_{i_k})$ forms a monotone sequence. Moreover for every subset of $k-1$ such $\delta_j$'s, i.e.\ $\delta_{j_1},\delta_{j_2},\dots,\delta_{j_{k-1}}$, there are $k$ vertices $v_{i_1},\dots,v_{i_k}$ such that $\delta(v_{i_t},v_{i_{t+1}}) = \delta_{j_t}$.

We now turn to the coloring of a 5-uniform hypergraph. Let $c > 0$ be the constant given by Lemma~\ref{lem:graphcolor} and let $U = \{0,1,\dots,\fl{2^{cn}}\}$ and $\phi: \binom{U}{2} \to \{\text{red}, \text{blue}\}$ be a 2-coloring of the pairs of $U$ satisfying the properties given in the lemma. Now let $N = 2^{\fl{2^{cn}}}$ and let $V = \{0,1,\dots,N-1\}$. In the following, we will use the coloring $\phi$ to define a red/blue coloring $\chi: \binom{V}{5} \to \{\text{red}, \text{blue}\}$ of the 5-tuples of $V$ such that $\chi$ produces at most $3$ red edges among any 6 vertices and $\chi$ does not produce a blue copy of $K^{(5)}_{128n^4}$. This would imply that $r_5(6,4;n) > 2^{2^{c'n^{1/4}}}$ for some constant $c' > 0$. 

For $v_1,\dots,v_5 \in V$ with $v_1 < v_2 < \cdots < v_5$, let $\dd_i = \dd(v_i,v_{i+1})$. We set $\chi(v_1,\dots,v_5) = \text{red}$ if:
\begin{enumerate}
\item We have that $\de1,\de2,\de3,\de4$ are monotone and form a bad 4-tuple, that is, if $\de1 < \de2 < \de3 < \de4$ then:
\[\phi(\de1,\de2) = \phi(\de2,\de3) = \phi(\de2,\de4) = \text{red}, \qquad \phi(\de1,\de3) = \phi(\de1,\de4) = \phi(\de3,\de4) = \text{blue},\]
and if $\de1 > \de2 > \de3 > \de4$ then:
\[\phi(\de4,\de3) = \phi(\de3,\de2) = \phi(\de3,\de1) = \text{red}, \qquad \phi(\de4,\de2) = \phi(\de4,\de1) = \phi(\de2,\de1) = \text{blue}.\]
\item We have that $\de1 > \de2 < \de3 > \de4$, where $\de1,\de2,\de3,\de4$ are all distinct with $\de1 < \de3, \de2 > \de4$ and $\phi(\de1,\de4) = \text{red}$, $\phi(\de2,\de4) = \text{blue}$. The ordering can also be expressed as $\de3 > \de1 > \de2 > \de4$.
\item We have that $\de1 < \de2 > \de3 < \de4$, where $\de1,\de2,\de3,\de4$ are all distinct with $\de1 < \de3, \de2 > \de4$ and $\phi(\de1,\de4) = \text{red}$, $\phi(\de1,\de3) = \text{blue}$. The ordering can also be expressed as $\de2 > \de4 > \de3 > \de1$.
\item We have that $\de1 < \de2 > \de3 < \de4$ and $\de1 = \de4$. In other words, $\de2 > \de1 = \de4 > \de3$.
\end{enumerate}
Otherwise $\chi(v_1,\dots,v_5) = \text{blue}$.

\begin{figure}
\begin{subfigure}{.35\textwidth}
\centering
\begin{tikzpicture}[scale=.5]
\node at (1,1) {$\delta_4$};
\node at (2,1) {$\delta_3$};
\node at (3,1) {$\delta_2$};
\node at (4,1) {$\delta_1$};

\node (1) at (1,0) [vertex] {};
\node (2) at (2,0) [vertex] {};
\node (3) at (3,0) [vertex] {};
\node (4) at (4,0) [vertex] {};

\draw[ultra thick, red] (3) edge (4)
                    (2) edge (3)
                    (1) edge[bend left = 40] (3);
\draw[ultra thick, blue] (2) edge[bend left = 40] (4)
                    (1) edge (2)
                    (1) edge[bend right = 30] (4);
                    
\draw (-.5,-1) node {$v_1$:} (1,-1) node {0} (2,-1) node {0} (3,-1) node {0} (4,-1) node {0};
\draw (-.5,-2) node {$v_2$:} (1,-2) node {0} (2,-2) node {0} (3,-2) node {0} (4,-2) node {1};
\draw (-.5,-3) node {$v_3$:} (1,-3) node {0} (2,-3) node {0} (3,-3) node {1} (4,-3) node {1};
\draw (-.5,-4) node {$v_4$:} (1,-4) node {0} (2,-4) node {1} (3,-4) node {1} (4,-4) node {1};
\draw (-.5,-5) node {$v_5$:} (1,-5) node {1} (2,-5) node {1} (3,-5) node {1} (4,-5) node {1};
\end{tikzpicture}
\begin{tikzpicture}[scale=.5]
\node at (1,1) {$\delta_1$};
\node at (2,1) {$\delta_2$};
\node at (3,1) {$\delta_3$};
\node at (4,1) {$\delta_4$};

\node (1) at (1,0) [vertex] {};
\node (2) at (2,0) [vertex] {};
\node (3) at (3,0) [vertex] {};
\node (4) at (4,0) [vertex] {};

\draw[ultra thick, red] (3) edge (4)
                    (2) edge (3)
                    (1) edge[bend left = 40] (3);
\draw[ultra thick, blue] (2) edge[bend left = 40] (4)
                    (1) edge (2)
                    (1) edge[bend right = 30] (4);
                    
\draw (1,-1) node {0} (2,-1) node {0} (3,-1) node {0} (4,-1) node {0};
\draw (1,-2) node {1} (2,-2) node {0} (3,-2) node {0} (4,-2) node {0};
\draw (1,-3) node {1} (2,-3) node {1} (3,-3) node {0} (4,-3) node {0};
\draw (1,-4) node {1} (2,-4) node {1} (3,-4) node {1} (4,-4) node {0};
\draw (1,-5) node {1} (2,-5) node {1} (3,-5) node {1} (4,-5) node {1};
\end{tikzpicture}
\subcaption{Monotone}
\end{subfigure}\begin{subfigure}{.2\textwidth}
\centering
\begin{tikzpicture}[scale=.5]
\node at (1,1) {$\delta_3$};
\node at (2,1) {$\delta_1$};
\node at (3,1) {$\delta_2$};
\node at (4,1) {$\delta_4$};

\node (1) at (1,0) [vertex] {};
\node (2) at (2,0) [vertex] {};
\node (3) at (3,0) [vertex] {};
\node (4) at (4,0) [vertex] {};

\draw[ultra thick, red] (2) edge[bend left = 40] (4);
\draw[ultra thick, blue] (3) edge (4);

\draw (1,-1) node {0} (2,-1) node {0} (3,-1) node {0} (4,-1) node {0};
\draw (1,-2) node {0} (2,-2) node {1} (3,-2) node {0} (4,-2) node {0};
\draw (1,-3) node {0} (2,-3) node {1} (3,-3) node {1} (4,-3) node {0};
\draw (1,-4) node {1} (2,-4) node {0} (3,-4) node {0} (4,-4) node {0};
\draw (1,-5) node {1} (2,-5) node {0} (3,-5) node {0} (4,-5) node {1};
\end{tikzpicture}
\caption{$\de3 > \de1 > \de2 > \de4$}
\end{subfigure}
\begin{subfigure}{.2\textwidth}
\centering
\begin{tikzpicture}[scale=.5]
\node at (1,1) {$\delta_2$};
\node at (2,1) {$\delta_4$};
\node at (3,1) {$\delta_3$};
\node at (4,1) {$\delta_1$};

\node (1) at (1,0) [vertex] {};
\node (2) at (2,0) [vertex] {};
\node (3) at (3,0) [vertex] {};
\node (4) at (4,0) [vertex] {};

\draw[ultra thick, red] (2) edge[bend left = 40] (4);
\draw[ultra thick, blue] (3) edge (4);

\draw (1,-1) node {0} (2,-1) node {0} (3,-1) node {0} (4,-1) node {0};
\draw (1,-2) node {0} (2,-2) node {0} (3,-2) node {0} (4,-2) node {1};
\draw (1,-3) node {1} (2,-3) node {0} (3,-3) node {0} (4,-3) node {0};
\draw (1,-4) node {1} (2,-4) node {0} (3,-4) node {1} (4,-4) node {0};
\draw (1,-5) node {1} (2,-5) node {1} (3,-5) node {0} (4,-5) node {0};
\end{tikzpicture}
\caption{$\de2 > \de4 > \de3 > \de1$}
\end{subfigure}\begin{subfigure}{.24\textwidth}
\centering
\begin{tikzpicture}[scale=.5]
\node at (.5,1) {$\delta_2$};
\node at (2,1) {$\delta_1/\delta_4$};
\node at (3.5,1) {$\delta_3$};

\node (1) at (1,0) [vertex] {};
\node (2) at (2,0) [vertex] {};
\node (3) at (3,0) [vertex] {};

\draw (1,-1) node {0} (2,-1) node {0} (3,-1) node {0};
\draw (1,-2) node {0} (2,-2) node {1} (3,-2) node {0};
\draw (1,-3) node {1} (2,-3) node {0} (3,-3) node {0};
\draw (1,-4) node {1} (2,-4) node {0} (3,-4) node {1};
\draw (1,-5) node {1} (2,-5) node {1} (3,-5) node {0};
\end{tikzpicture}
\caption{$\delta_2 > \delta_1 = \delta_4 > \delta_3$}
\end{subfigure}
\caption{Examples of $v_1 < v_2 < v_3 < v_4 < v_5$ and $\delta_i = \delta(v_i,v_{i+1})$ for $i \in [4]$ such that $\chi(v_1,\dots,v_5)$ is red. Each $v_i$ is represented in binary with the left-most entry corresponding to the most significant bit.}
\end{figure}
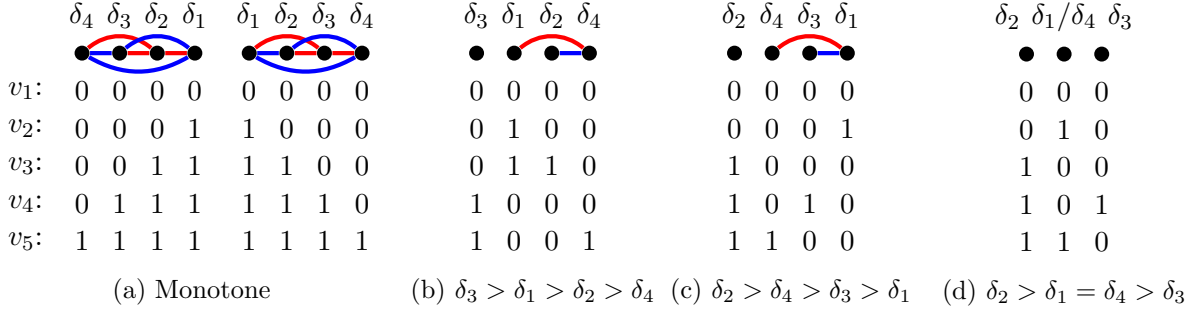

Assume for the sake of contradiction that there are at least 4 red edges among some 6 vertices. Let these vertices be $v_1,\dots,v_6$ where $v_1 < v_2 < \cdots < v_6$ and let $\dd_i = \dd(v_i,v_{i+1})$. Let $e_i = \{v_1,\dots,v_6\} \setminus \{v_i\}$. Let $\dd(e_i)$ be the resulting sequence of $\dd$'s. In particular, for $i = 1$, $\dd(e_1) = (\de2,\de3,\de4,\de5)$. For $2 \leq i \leq 5$, $\dd(e_i) = (\de1, \dots, \dd(v_{i-1},v_{i+1}), \dots, \de5)$. For $i = 6$, $\dd(e_6) = (\de1,\de2,\de3,\de4)$. In the following we will often use that if $2 \leq i \leq 5$, then $\delta(v_{i-1},v_{i+1}) = \max(\de{i-1},\de{i})$ by Property II.

For convenience, if inequalities are known between consecutive $\delta$'s, this will be indicated in the sequence by replacing the comma with the respective sign. For instance, assume that $\de1 < \de2 > \de3 < \de4 > \de5$. Then since $\dd(e_1) = (\de2,\de3,\de4,\de5)$ has $\de2 > \de3 < \de4 > \de5$, we will write
\[\dd(e_1) = (\de2 > \de3 < \de4 > \de5).\]
Similarly, if not all inequalities are known, as in $\dd(e_3)$, we write,
\[\dd(e_3) = (\de1 < \de2 \com \de4 > \de5).\]

Now we will consider cases depending on the ordering of $\de1,\dots,\de5$, and we will further split into subcases by taking an ordering and reversing it. There are 16 possible orderings so we will have 8 cases in what follows. 

\textit{Case 1a}: Suppose $\de1 > \de2 < \de3 > \de4 < \de5$. This implies that
\begin{align*}
\dd(e_1) &= (\de2 < \de3 > \de4 < \de5),\\
\dd(e_2) = \dd(e_3) &= (\de1 \com \de3 > \de4 < \de5),\\
\dd(e_4) = \dd(e_5) &= (\de1 > \de2 < \de3 \com \de5),\\
\dd(e_6) &= (\de1 > \de2 < \de3 > \de4).
\end{align*}
In particular, note that at least one of $e_4,e_5,e_6$ must be red so we must have that $\delta_1 < \delta_3$ and $\delta_2 > \de4$. However, since $\de1 > \de2 > \de4$, note that $e_1$ is only red if $\de2 = \de5$ and similarly $e_2,e_3$ are only red if $\de1 = \de5$. Since these cannot happen simultaneously, there is at least one blue edge among these three edges. Thus, we must have that $e_4, e_5$ are also red to avoid having three blue edges, making $\de2 > \de5$ (and $\de3 > \de5$). However, then $\de1 > \de2 > \de5$ so none of $e_1,e_2,e_3$ are red and thus there are at most 3 red edges.

\textit{Case 1b}: Suppose $\de1 < \de2 > \de3 < \de4 > \de5$. This implies that
\begin{align*}
\dd(e_1) = \dd(e_2) &= (\de2 > \de3 < \de4 > \de5),\\
\dd(e_3) = \dd(e_4) &= (\de1 < \de2 \com \de4 > \de5),\\
\dd(e_5) = \dd(e_6) &= (\de1 < \de2 > \de3 < \de4).
\end{align*}
Note that $e_3,e_4$ are blue so we must have that all of $e_1,e_2,e_5,e_6$ are red. If $e_5,e_6$ are red, then regardless of which rule applies, $\de2 > \de4$ and thus $e_1,e_2$ are blue, so there are at most 2 red edges.

\textit{Case 2a}: Suppose $\de1 > \de2 > \de3 < \de4 > \de5$. This implies that
\begin{align*}
\dd(e_1) &= (\de2 > \de3 < \de4 > \de5),\\
\dd(e_2) &= (\de1 > \de3 < \de4 > \de5),\\
\dd(e_3) = \dd(e_4) &= (\de1 > \de2 \com \de4 > \de5),\\
\dd(e_5) = \dd(e_6) &= (\de1 > \de2 > \de3 < \de4).
\end{align*}
Note that $e_5,e_6$ are blue so that all of $e_1,\dots,e_4$ are red. Since $e_1$ is red, we must have that $\de2 < \de4$, so $\dd(e_i)$ are ordered as in the second condition for red edges for all $i\in[4]$. Thus, $e_1$ implies that $\phi(\de2,\de5) = \text{red}$ while $e_3$ implies that $\phi(\de2,\de5) = \text{blue}$, a contradiction.

\textit{Case 2b}: Suppose $\de1 < \de2 < \de3 > \de4 < \de5$. This implies that
\begin{align*}
\dd(e_1) = \dd(e_2) &= (\de2 < \de3 > \de4 < \de5),\\
\dd(e_3) &= (\de1 < \de3 > \de4 < \de5),\\
\dd(e_4) = \dd(e_5) &= (\de1 < \de2 < \de3 \com \de5),\\
\dd(e_6) &= (\de1 < \de2 < \de3 > \de4).
\end{align*}
Since $e_6$ is blue, in order to have at least 4 red edges, we must have that $e_4,e_5$ are red. Thus $\de3 < \de5$. However, then for $e_1,e_2$ to be red, we must have that $\de2 = \de5$, which is impossible since $\de2 < \de5$. Thus, there are at most 3 red edges here.

\textit{Case 3a}: Suppose $\de1 > \de2 < \de3 > \de4 > \de5$. This implies that
\begin{align*}
\dd(e_1) &= (\de2 < \de3 > \de4 > \de5),\\
\dd(e_2) = \dd(e_3) &= (\de1 \com\de3 > \de4 > \de5),\\
\dd(e_4) &= (\de1 > \de2 < \de3 > \de5),\\
\dd(e_5) = \dd(e_6) &= (\de1 > \de2 < \de3 > \de4).
\end{align*}
Since $e_1$ is blue, we must have that $e_5,e_6$ are red and thus $\de1 < \de3$. However, we also must have $e_2,e_3$ are red and thus $\de1 > \de3$, a contradiction.

\textit{Case 3b}: Suppose $\de1 < \de2 > \de3 < \de4 < \de5$. This implies that
\begin{align*}
\dd(e_1) = \dd(e_2) &= (\de2 > \de3 < \de4 < \de5),\\
\dd(e_3) = \dd(e_4) &= (\de1 < \de2 \com \de4 < \de5),\\
\dd(e_5) &= (\de1 < \de2 > \de3 < \de5),\\
\dd(e_6) &= (\de1 < \de2 > \de3 < \de4).
\end{align*}
Since $e_1,e_2$ are blue, we must have that the remaining edges are red. If $\de2 < \de4$, then $e_6$ is blue. Otherwise $\de2 > \de4$. First if $\de1 = \de4$ then $e_3,e_4$ are blue. Thus, for $e_6$ to be red, we have that $\de1 < \de3$, which implies that $\de1 < \de4 < \de5$. From $e_3$ being red, we find that $\de2 > \de5$ as well. We then have that $\phi(\de1,\de4) = \text{red}$ from $e_6$ while $\phi(\de1, \de4) = \text{blue}$ from $e_3$, a contradiction.

\textit{Case 4a}: Suppose $\de1 > \de2 < \de3 < \de4 > \de5$. This implies that
\begin{align*}
\dd(e_1) &= (\de2 < \de3 < \de4 > \de5),\\
\dd(e_5) = \dd(e_6) &= (\de1 > \de2 < \de3 < \de4).
\end{align*}
so we have at least 3 blue edges.

\textit{Case 4b}: Suppose $\de1 < \de2 > \de3 > \de4 < \de5$. This implies that
\begin{align*}
\dd(e_1) = \dd(e_2) &= (\de2 > \de3 > \de4 < \de5),\\
\dd(e_6) &= (\de1 < \de2 > \de3 > \de4).
\end{align*}
so we have at least 3 blue edges.


\textit{Case 5}: Suppose $\de1 > \de2 < \de3 < \de4 < \de5$ or $\de1 < \de2 > \de3 > \de4 > \de5$. In the first case, each of $\dd(e_4), \dd(e_5), \dd(e_6)$ is in the form $\de1 > \de2 < \de{i} < \de{j}$ where $i,j \in \{3,4,5\}$, so these are blue. In the second case, each of $\dd(e_4), \dd(e_5), \dd(e_6)$ is in the form $\de1 < \de2 > \de{i} > \de{j}$ where $i, j \in\{3,4,5\}$, so these are blue.

\textit{Case 6}: Suppose $\de1 > \de2 > \de3 < \de4 < \de5$ or $\de1 < \de2 < \de3 > \de4 > \de5$. In the first case, 
\begin{align*}
\dd(e_1) &= (\de2 > \de3 < \de4 < \de5),\\
\dd(e_2) &= (\de1 > \de3 < \de4 < \de5),\\ 
\dd(e_6) &= (\de1 > \de2 > \de3 < \de4).
\end{align*}
so there are at least 3 blue edges. In the second case, $\dd(e_1), \dd(e_2)$ are both $\de2 < \de3 > \de4 > \de5$ and thus blue. Similarly, $\dd(e_6) = \de1 < \de2 < \de3 > \de4$, so there are at least 3 blue edges.

\textit{Case 7}: Suppose $\de1 > \de2 > \de3 > \de4 < \de5$ or $\de1 < \de2 < \de3 < \de4 > \de5$. In the first case, each of $\dd(e_1), \dd(e_2), \dd(e_3)$ is in the form $\de{i} > \de{j} > \de4 < \de5$ for $i,j \in [3]$ and thus blue. In the second case, each of $\dd(e_1), \dd(e_2), \dd(e_3)$ is in the form $\de{i} < \de{j} < \de4 > \de5$ for $i,j \in [3]$ and thus blue.

\textit{Case 8a}: Suppose $\de1 > \de2 > \de3 > \de4 > \de5$. This implies that
\begin{align*}
\dd(e_1) &= (\de2 > \de3 > \de4 > \de5),\\
\dd(e_2) &= (\de1 > \de3 > \de4 > \de5),\\
\dd(e_3) &= (\de1 > \de2 > \de4 > \de5),\\
\dd(e_4) &= (\de1 > \de2 > \de3 > \de5),\\
\dd(e_5) = \dd(e_6) &= (\de1 > \de2 > \de3 > \de4).
\end{align*}
First if $e_5, e_6$ are red, then $\phi(\de4, \de1) = \text{blue}$ implies that $e_2, e_3$ are blue, and $\phi(\de4, \de2) = \text{blue}$ implies that $e_1$ is blue, a contradiction. Thus, $e_5, e_6$ are blue and $e_1$ must be red but then $\phi(\de5, \de3) = \text{blue}$ implies that $e_4$ is blue, a contradiction.

\textit{Case 8b}: Suppose $\de1 < \de2 < \de3 < \de4 < \de5$. This implies that
\begin{align*}
\dd(e_1) = \dd(e_2) &= (\de2 < \de3 < \de4 < \de5),\\
\dd(e_3) &= (\de1 < \de3 < \de4 < \de5),\\
\dd(e_4) &= (\de1 < \de2 < \de4 < \de5),\\
\dd(e_5) &= (\de1 < \de2 < \de3 < \de5),\\
\dd(e_6) &= (\de1 < \de2 < \de3 < \de4).
\end{align*}
If $e_1,e_2$ are red, then $\phi(\de2, \de5) = \text{blue}$ implies that $e_4,e_5$ are blue and $\phi(\de2, \de4) = \text{blue}$ implies that $e_6$ is blue, a contradiction. Thus, $e_1, e_2$ are blue and $e_6$ must be red but then $\phi(\de1,\de3) = \text{blue}$ implies that $e_3$ is blue, a contradiction.

Thus, for every 6 vertices in $V = \{0,1,\ldots, 2^{\lfloor 2^{cn}\rfloor} - 1\}$, $\chi$ produces at most 3 red edges among them.


Now, we show that there is no blue $K^{(5)}_{128n^4}$ in coloring $\chi$. We first make the following definitions. Given a sequence $\{a_i\}_{i=1}^r \subseteq \mathbb{R}$ and $j \in \{2,\dots,r-1\}$, we say that $a_j$ is a \textit{local minimum} if $a_{j-1} > a_j < a_{j+1}$, a \textit{local maximum} if $a_{j-1} < a_j > a_{j+1}$, and a \textit{local extremum} if it is either a local minimum or local maximum. In particular, when looking at some set of vertices $\{v_1,\dots,v_s\}$ where $v_1 < v_2 < \cdots < v_s$ and considering the sequence $\{\delta(v_i,v_{i+1})\}_{i=1}^{s-1}$, by Property I, $\delta(v_j,v_{j+1}) \neq \delta(v_{j+1},v_{j+2})$ for every $j$, so every nonmonotone sequence will have local extrema. 

Set $m = 128n^4$ and consider vertices $v_1,\dots,v_m \in V$ such that $v_1 < v_2 < \cdots < v_m$. Assume for the sake of contradiction that these $m$ vertices correspond to a blue clique in the coloring $\chi$. Again, let $\de{i} = \dd(v_i,v_{i+1})$. We first note the following lemma.

\begin{lemma}
\label{lem:monotone}
There is no monotone subsequence $\{\de{k_\ell}\}_{\ell=1}^{n} \subset \{\de{i}\}_{i=1}^{m-1}$ such that for any $a,b,c,d \in [n]$ with $a < b < c < d$, there exists $u_1,u_2,u_3,u_4,u_5 \subset \{v_1,\dots, v_m\}$ such that $\dd(u_1,\dots,u_5) = \{\de{k_a},\de{k_b},\de{k_c},\de{k_d}\}$.
\end{lemma}
\begin{proof}Indeed, if such a monotone subsequence existed, then as $\chi(u_1,\dots,u_5) = \text{blue}$, we have that $\{\de{k_\ell}\}_{\ell=1}^{n}$ would form an $n$-set with no bad 4-tuple in the graph coloring $\phi$, a contradiction.\end{proof}

From this, we note that there is no integer $j \in [m-n+1]$ such that the sequence $\{\de{i}\}_{i=j}^{j+n-1}$ is monotone. Otherwise, by Property IV, we have that for any length 4 subsequence $\{\dei{1},\dei{2},\dei{3},\dei{4}\} \subset \{\de{i}\}_{i=j}^{j+n-1}$, there is a 5-tuple $e \subset \{v_1,\dots,v_m\}$ such that $\dd(e)$ corresponds to this monotone sequence. From here, we apply Lemma~\ref{lem:monotone} to get a contradiction. Thus, we can find a sequence of consecutive local extrema and from this extract a sequence of local maxima $\dei1,\dots,\dei{32n^3}$.

We now restrict our attention to this sequence of local maxima $(\dei1,\dots,\dei{32n^3})$.  Note that any two local maxima are distinct: assume for the sake of contradiction that we have maxima $\dei{j} = \dei{k}$ where $j<k$. First consider if there is no $\delta_\ell$ for $i_j < \ell < i_k$ such that $\delta_\ell > \dei{j} = \dei{k}$. Then, $\dd(v_{i_j},v_{i_k}) = \dei{j} = \dei{k} = \dd(v_{i_k},v_{i_k+1})$, a contradiction of Property I. Otherwise, there exists $i_j < \ell < i_k$ such that $\delta_\ell > \dei{j} = \dei{k}$.  By letting $\ell$ correspond to the maximum $\de{\ell}$ in this range, we have 
\[\dd(v_{i_j},v_{i_j+1},v_{i_k-1},v_{i_k},v_{i_k+1}) = (\dei{j} < \de{\ell} > \de{i_k-1} < \dei{k}),\]
which implies that $\chi(v_{i_j},v_{i_j+1},v_{i_k-1},v_{i_k},v_{i_k+1}) = \text{red}$ as $\dei{j} = \dei{k}$, contradiction.

Moreover, there is no $j \in [32n^3-n+1]$ such that the sequence $\{\dei{k}\}_{k=j}^{j+n-1}$ is monotone. If there is such $j$ and the sequence is increasing, for any $a,b,c,d \in \{j,j+1,\dots,j+n-1\}$ with $a < b < c < d$, then
\[\dd(v_{i_{a}},v_{i_{a}+1},v_{i_{b}+1},v_{i_{c}+1},v_{i_{d}+1}) = (\dei{a} < \dei{b} < \dei{c} < \dei{d}).\]
This follows by Property II; in particular, if there exists $\ell$ such that $i_{a}+1 \leq \ell < i_{b}+1$ and $\de{\ell} > \dei{b}$, then there must exist some greater local maxima between $\dei{a}$ and $\dei{b}$, a contradiction of the monotonicity of $\{\dei{k}\}_{k=j}^{j+n-1}$, as these are consecutive local maxima. Thus, by Lemma~\ref{lem:monotone}, we have a contradiction.

Similarly, if the sequence is decreasing, consider any $a,b,c,d \in \{j,j+1,\dots,j+n-1\}$ with $a < b < c < d$. Then
\[\dd(v_{i_{a}},v_{i_{b}},v_{i_{c}},v_{i_{d}},v_{i_{d}+1}) = (\dei{a} > \dei{b} > \dei{c} > \dei{d}).\]
As with the above, we apply Lemma~\ref{lem:monotone} to derive a contradiction.

Thus, within the sequence $(\dei{1},\dei{2},\dots,\dei{32n^3})$, we can find a subsequence of consecutive local extrema $\dej{1},\dots,\dej{16n^2}$, where $\dej{1},\dej{3},\dots,\dej{16n^2-1}$ are local maxima and $\dej{2},\dej{4},\dots,\dej{16n^2}$ are local minima (with respect to the sequence $\dei{1},\dei{2},\dots,\dei{32n^3}$).   

We now claim that  there exists $k \in \{4n+1,4n+2,\dots,16n^2-4n\}$ such that $\dej{\ell} < \dej{k}$ if $k-4n \leq \ell \leq k+4n$ and $\ell \neq k$. Assume for the sake of contradiction that this is not the case. We then recursively build the following sets $S_r,T_r$. Start with $S_0 = T_0 = \varnothing, \sigma_0 = 0, \tau_0 = 16n^2+1$. At each step $r$,
\begin{enumerate}
\item $\sigma_r = 0$ if $S_r$ is empty and $\sigma_r = \max(S_r)$ otherwise. Similarly, $\tau_r = 16n^2+1$ if $T_r$ is empty and $\tau_r = \min(T_r)$ otherwise.
\item If $s \in S_r$ and $s < \ell < \tau_r$, then $\dej{s} > \dej{\ell}$. Similarly if $t \in T_r$ and $\sigma_r < \ell < t$, then $\dej{t} > \dej{\ell}$.
\item $\abs{S_r} + \abs{T_r} = r$ and $\tau_r - \sigma_r \geq 16n^2 - 4nr$.
\end{enumerate}
Note that these properties hold for $r = 0$ by definition. Now assume that we have $S_r, T_r, \sigma_r, \tau_r$ satisfying the desired properties for some $r < 2n$. Note that by the properties, we have that 
\[\tau_r - \sigma_r \geq 16n^2 - 4nr > 16n^2 - 8n^2 \geq 8n^2 > 0.\]
Consider $\sigma_r < k < \tau_r$ such that $\dej{k} = \max_{\sigma_r < \ell < \tau_r}\dej{\ell}$. If $k - \sigma_r > 4n$ and $\tau_r - k > 4n$, then $k$ would satisfy that $\dej{\ell} < \dej{k}$ if $k-4n \leq \ell \leq k+4n$ and $\ell \neq k$, a contradiction. Now if $k - \sigma_r \leq 4n$, set 
\[S_{r+1} = S_r \cup \{k\},\quad T_{r+1} = T_r,\quad \sigma_{r+1} = k,\quad \tau_{r+1} = \tau_r.\]
Then, the first property holds by definition. The second property holds for every $s \in S_r, t \in T_r$ by assumption, and it holds for $k \in S_{r+1}$ since $\dej{k} = \max_{\sigma_r < \ell < \tau_r}\dej{\ell}$. The first part of the third property clearly holds and
\[\tau_{r+1} - \sigma_{r+1} = \tau_r - k \geq \tau_r - \sigma_r - 4n \geq 16n^2 - 4n(r+1).\]
Otherwise if $\tau_r - k \leq 4n$, set
\[S_{r+1} = S_r, \quad T_{r+1} = T_r \cup \{k\}, \quad \sigma_{r+1} = \sigma_r, \quad \tau_{r+1} = k.\]
By the same reasoning, the three properties hold as desired. Thus, we can construct these sets while $r \leq 2n$.

Now, consider $S_{2n}, T_{2n}$. Since $\abs{S_{2n}} + \abs{T_{2n}} = 2n$, at least one of these sets has size at least $n$. If $\abs{S_{2n}} \geq n$, consider $\{s_1, \dots, s_n\} \subseteq S_{2n}$ where $i < j \Rightarrow s_i < s_j$. Then, since $\min(T_{2n}) > \max(S_{2n})$ by Property 3 and 1, by Property 2 we have
\[\dej{s_1} > \dej{s_2} > \cdots > \dej{s_n}.\]
In particular, Property 2 implies that for $a,b,c,d\in [n]$ and $a < b < c < d$,
\[\dd(v_{j_{s_a}}, v_{j_{s_b}}, v_{j_{s_c}}, v_{j_{s_d}},v_{j_{s_d}+1}) = (\dej{s_a} > \dej{s_b} > \dej{s_c} > \dej{s_d}),\]
and thus, by Lemma~\ref{lem:monotone}, we have a contradiction. If instead $\abs{T_{2n}} \geq n$, a similar argument shows that we derive a contradiction. Thus, such a $k$ exists and note that in particular $k$ must be odd.

Order the set of local minima $\{\dej{k-4n+1},\dej{k-4n+3},\dots,\dej{k+4n-1}\}$ in increasing order as $\gamma_1,\dots,\gamma_{4n}$. Let 
\[A' = \{\dej{k-4n+1},\dej{k-4n+3},\dots,\dej{k-1}\}\text{ and }B' = \{\dej{k+1},\dej{k+3},\dots,\dej{k+4n-1}\}.\]
Note that since $A',B'$ partition $\{\dej{k-4n+1},\dej{k-4n+3},\dots,\dej{k+4n-1}\}$, either $\abs{A' \cap \{\gamma_1,\dots,\gamma_{2n}\}} \geq n$ or $\abs{B' \cap \{\gamma_1,\dots,\gamma_{2n}\}} \geq n$. Without loss of generality, we assume that the former occurs since a symmetric argument would follow otherwise. Then, we also have that $\abs{B' \cap \{\gamma_{2n+1},\dots,\gamma_{4n}\}} \geq n$. Set
\[A = A' \cap \{\gamma_1,\dots,\gamma_{2n}\} \text{ and }B = B' \cap \{\gamma_{2n+1},\dots,\gamma_{4n}\}.\]

Let $a \in A$ and $b \in B$. By definition, $\dej{a} < \dej{b}$, and note that $b < k+4n \Rightarrow b+1 \leq k+4n$, so
\[\dd(v_{j_{a}}, v_{j_{a}+1}, v_{j_{b}}, v_{j_{b}+1}, v_{j_{b+1}+1}) = (\dej{a} < \dej{k} > \dej{b} < \dej{b+1}),\]
where $\dej{k} > \dej{b+1}$ by definition. Since $$\chi(v_{j_{a}}, v_{j_{a}+1}, v_{j_{b}}, v_{j_{b}+1}, v_{j_{b+1}+1}) = \textnormal{blue},$$ we cannot have both $\phi(\dej{a},\dej{b+1}) = \text{red}$ and $\phi(\dej{a},\dej{b}) = \text{blue}$. 
Finally, restricting to any $n$ elements of $A,B$ and letting
\[C = \{\dej{b+1}: \dej{b} \in B\},\]
and defining $f: B \to C$ via $\dej{b} \mapsto \dej{b+1}$, we obtain 3 disjoint $n$-sets with precisely the structure avoided in the graph coloring $\phi$, a contradiction.

Thus, $\chi$ does not produce a blue $K^{(5)}_{128n^4}$ on $V$.  $\hfill\square$

\section{Concluding remarks}

We have determined the tower growth rate for $r_k(k+1,k-1;n)$.  Thus, the only problem remaining for the Erd\H os-Hajnal hypergraph Ramsey conjecture, is to determine the tower growth rate for $r_k(k+1,k;n)$.  

Let us remark that similar arguments show that $r_5(6,5;4n^2) > 2^{r_4(5,4;n)-1}$. To define such a coloring, let $N = r_4(5,4;n)-1$ and let $\varphi$ be a red/blue coloring of the 4-tuples of $\{0,\dots,N-1\}$ such that there are there are at most 3 red edges among every 5 vertices and there is no blue clique of size $n$. We then color the 5-tuples of $V = \{0,1,\dots,2^N-1\}$ so that $\chi$ produces at most 4 red edges among any 6 vertices and $\chi$ does not produce a blue clique of size $4n^2$. For vertices $v_1,\dots,v_5$ with $v_1 < v_2 < \cdots < v_5$, let $\de{i} = \dd(v_i,v_{i+1})$. We set $\chi(v_1,\dots,v_5) = \text{red}$ if:
\begin{enumerate}
\item We have that $\de1,\de2,\de3,\de4$ are monotone and $\varphi(\de1,\de2,\de3,\de4) = \text{red}$.
\item We have that $\de1 > \de2 < \de3 > \de4$ and $\de1 < \de3$.
\end{enumerate}
Together with Lemma \ref{stepk}, showing that $r_4(5,4;n)$ grows double exponential in a power of $n$ would thus show that $r_k(k+1,k;n) = \twr_{k-1}(n^{\Theta(1)}).$

\end{document}